\documentclass[reqno]{amsart}%
\usepackage{amsmath}
\usepackage{amsfonts}
\usepackage{amssymb}
\usepackage{graphicx}%
\providecommand{\U}[1]{\protect\rule{.1in}{.1in}}
\newtheorem{theorem}{Theorem}[section]
\theoremstyle{plain}

\newtheorem{lemma}[theorem]{Lemma}

\newtheorem{proposition}[theorem]{Proposition}
\newtheorem{remark}[theorem]{Remark}

\numberwithin{equation}{section}
\begin{document}
\title[Existence and uniqueness for...]{Existence and uniqueness for nonlinear anisotropic elliptic equations}
\author{R. Di Nardo}
\address{Dipartimento di Matematica e Applicazioni \textquotedblright R.
Caccioppoli\textquotedblright, Universit\'{a} degli Studi di Napoli Federico
II, \\
Complesso Universitario Monte S. Angelo, via Cintia - Napoli}
\email{rosaria.dinardo@unina.it }
\author{F. Feo}
\address{Dipartimento per le Tecnologie\\
Universit\'{a} degli Studi di Napoli Parthenope\\
Centro Direzionale, Isola C4 - Napoli}
\email{filomena.feo@uniparthenope.it}

\begin{abstract}
We study the existence and uniqueness for weak solutions to some classes of
anisotropic elliptic Dirichlet problems with data belonging to the natural
dual space.

\end{abstract}
\maketitle




\numberwithin{equation}{section}

\section{Introduction}

In the present paper we study the existence and uniqueness of weak solutions
to some classes of anisotropic elliptic equations with homogeneous Dirichlet
boundary conditions.

\noindent Let us consider the following model problem%

\begin{equation}
\left\{
\begin{array}
[c]{lll}%
-\partial_{x_{i}}(a_{i}(x,u)(\varepsilon+\left\vert \partial_{x_{i}%
}u\right\vert ^{2})^{\!\frac{p_{i}-2}{2}\!}\partial_{x_{i}}u)=\!f-\partial
_{x_{i}}g_{i} & \text{in} & \Omega\\
&  & \\
u=0 & \text{on} & \partial\Omega,
\end{array}
\right.  \label{prototipo}%
\end{equation}
where $\Omega$ is a bounded open subset of $\mathbb{R}^{N}$ with Lipschitz
continuous boundary, $N\geq2,$ $1<p_{1},\ldots,p_{N}<+\infty,$ $\varepsilon
\geq0,$ $a_{1},...,a_{N}$ are Carath\'{e}odory functions, $g_{1},..,g_{N}$ and
$f$ are functions belonging to suitable Lebesgue spaces.

The anisotropy of the problem is due to the growth in each partial derivative
controlled by different powers. The interest in studying such operators are
motivated by their applications to the mathematical modeling of physical and
mechanical processes in anisotropic continuous medium.

The existence and regularity of weak solutions or solutions in the
distributional sense to Problem (\ref{prototipo}) with $g_{i}\equiv0$ have
been studied in \cite{li} when $f\in L^{m}(\Omega)$ with $m>1$, in \cite{di
castro} when datum $f\ $belongs to Marcinkiewicz spaces and in \cite{boccardo
anisotropo} for measure data. In \cite{cianchi} a comparison theorem and the
derived a priori estimates are established via symmetrization methods. Some
uniqueness results for Problem (\ref{prototipo}) have been obtained in
\cite{chipot} for weak solutions and data belonging to the dual of the
anisotropic Sobolev space. Moreover when the datum $f$ is only integrable the
uniqueness of a renormalized solution is proved\ in \cite{olivier}.

As far as the uniqueness of a weak solution to Problem (\ref{prototipo}) is
concerned, when $\varepsilon=0$ and at least one $p_{i}$ is less or equals to
$2$ in this paper we obtain the same uniqueness result of \cite{chipot} by a
different method. Instead we improve the result of \cite{chipot} when
$\varepsilon>0$ and every $p_{i}$ is greater than $2.$ The main tools in our
proofs are Poincar\'{e} inequality and the embedding for the anisotropic
Sobolev spaces.

We also consider a class of anisotropic equations with a first order term,
whose prototype is
\begin{equation}
\left\{
\begin{array}
[c]{lll}%
\!-\partial_{x_{i}}\!(\!a(x,u)(\varepsilon+\left\vert \partial_{x_{i}%
}u\right\vert ^{2})^{\!\frac{p_{i}-2}{2}\!}\!\partial_{x_{i}}u\!)\!+\!\overset
{N}{\underset{i=1}{\sum}}\!b_{i}\left\vert \partial_{x_{i}}u\right\vert
^{p_{i}-1}\!\!=\!f\!-\!\partial_{x_{i}}g_{i}\! & \!\text{in}\! & \!\Omega\!\\
\!u=0 & \!\text{on}\! & \!\partial\Omega,\!
\end{array}
\right.  \label{prototipo 2}%
\end{equation}
where $b_{i}$ belong to suitable Lebesgue spaces for $i=1,..,N.$ To our
knowledge, such a problem is not still deeply studied. As far as the existence
of a weak solution is concerned, the presence of a lower order term produces a
lack of coerciveness, which does not allow to use the classical methods. Here
we prove the existence of a weak solution to Problem (\ref{prototipo 2}). As
usual the main step in the proof is an a priori estimate. In order to avoid
the assumption on smallness of the norm of the coefficients $b_{i}$ (that
implies the coerciveness of the operator), we adapt the method used in
\cite{bottaro} (see also \cite{betta CRAS}, \cite{livier esistenza}, \cite{gm}
and \cite{gm2}), which consists in splitting the domain $\Omega$ in a finite
number of small domains $\Omega_{1},..,\Omega_{t}$ in such a way to have small
norms of the coefficients on $\Omega_{\sigma}$ for $\sigma=1,..,t$. Finally we
consider a different class of anisotropic operator, whose lower order term
satisfy a Lipschitz condition in order to obtain same uniqueness results
following the idea of \cite{alvino}.

Problems (\ref{prototipo}) and (\ref{prototipo 2}) have been studied in the
isotropic case by many authors. We just mention some of these papers:
\cite{betta-mercaldo}, \cite{Chiacchio}, \cite{del vecchio} and \cite{del
vecchio post} for existence and regularity of weak solutions and
\cite{alvino}, \cite{casado}, \cite{guibe-mercaldo} and \cite{porretta} for
the uniqueness (see also the references therein).

The paper is organized as follows. In \S \ 2 we recall the standard framework
of anisotropic Sobolev spaces, we detail the assumptions and we give the
notion of weak solution. In \S \ 3 we study the case of strongly monotone
operator. Finally we investigate Problem (\ref{prototipo 2}): in \S 4.1 we
prove the existence of at least a weak solution and in \S \ 4.2 we prove some
uniqueness results.

\section{Definitions, assumptions and preliminaries results}

\subsection{Anisotropic Sobolev spaces}

Let $\Omega$ be a bounded open subset of $\mathbb{R}^{N}$ ($N\geq2$) with
Lipschitz continuous boundary and let $1<p_{1},\ldots,p_{N}<\infty$ be $N$
real numbers. The anisotropic space (see e.g. \cite{troisi})
\[
W^{1,\overrightarrow{p}}(\Omega)=\left\{  u\in W^{1,1}(\Omega):\partial
_{x_{i}}u\in L^{p_{i}}(\Omega),i=1,...,N\right\}
\]
is a Banach space with respect to norm $\left\Vert u\right\Vert
_{W^{1,\overrightarrow{p}}(\Omega)}=\left\Vert u\right\Vert _{L^{1}(\Omega
)}+\overset{N}{\underset{i=1}{\sum}}\left\Vert \partial_{x_{i}}u\right\Vert
_{L^{p_{i}}(\Omega)}.$ The space $W_{0}^{1,\overrightarrow{p}}(\Omega)$ is the
closure of $C_{0}^{\infty}(\Omega)$ with respect to this norm$.$

\noindent We recall a Poincar\'{e}-type inequality. Let $u\in W_{0}%
^{1,\overrightarrow{p}}(\Omega)$, then for every $q\geq1$ there exists a
constant $C_{P}$ (depending on $q$ and $i$) such that (see \cite{fargala})
\begin{equation}
\left\Vert u\right\Vert _{L^{q}(\Omega)}\leq\ C_{P}\left\Vert \partial_{x_{i}%
}u\right\Vert _{L^{q}(\Omega)}\text{ \ for }i=1,\ldots,N. \label{dis poincare}%
\end{equation}
Moreover a Sobolev-type inequality holds. Let us denote by $\overline{p}$ the
harmonic mean of these numbers, i.e. $\frac{1}{\overline{p}}=\frac{1}%
{N}\overset{N}{\underset{i=1}{\sum}}\frac{1}{p_{i}}$. Let $u\in W_{0}%
^{1,\overrightarrow{p}}(\Omega)$, then there exists (see \cite{troisi}) a
constant $C_{S}$ such that%
\begin{equation}
\left\Vert u\right\Vert _{L^{q}(\Omega)}\leq C_{S}\overset{N}{\underset
{i=1}{{\displaystyle\prod}}}\left\Vert \partial_{x_{i}}u\right\Vert
_{L^{p_{i}}(\Omega)}^{\frac{1}{N}}, \label{sobolev 2}%
\end{equation}
where $q=\overline{p}^{\ast}=\frac{N\overline{p}}{N-\overline{p}}$ if
$\overline{p}<N$ or $q\in\left[  1,+\infty\right[  $ if $\overline{p}\geq N$.
On the right-hand side of (\ref{sobolev 2}) it is possible to replace the
geometric mean by the arithmetic mean: let $a_{1},...,a_{N}$ be positive
numbers$,$ it holds%
\begin{equation}
\overset{N}{\underset{i=1}{{\displaystyle\prod}}}a_{i}^{\frac{1}{N}}\leq
\frac{1}{N}\overset{N}{\underset{i=1}{\sum}}a_{i}, \label{geo-arti}%
\end{equation}
which implies by (\ref{sobolev 2})%
\begin{equation}
\left\Vert u\right\Vert _{L^{q}(\Omega)}\leq\frac{C_{S}}{N}\overset
{N}{\underset{i=1}{\sum}}\left\Vert \partial_{x_{i}}u\right\Vert _{L^{p_{i}%
}(\Omega)}. \label{sobolev}%
\end{equation}
When
\begin{equation}
\overline{p}<N \label{cond immersione}%
\end{equation}
hold, inequality (\ref{sobolev}) implies the continuous embedding of the space
$W_{0}^{1,\overrightarrow{p}}(\Omega)$ into $L^{q}(\Omega)$ for every
$q\in\lbrack1,\overline{p}^{\ast}]$. On the other hand the continuity of the
embedding $W_{0}^{1,\overrightarrow{p}}(\Omega)\subset L^{p_{+}}(\Omega)$ with
$p_{+}:=\max\{p_{1},...,p_{N}\}$ relies on inequality (\ref{dis poincare}). It
may happen that $\overline{p}^{\ast}<p_{+}$ if the exponents $p_{i}$ are not
closed enough, then $p_{\infty}:=\max\{\overline{p}^{\ast},p_{+}\}$ turns out
to be the critical exponent in the anisotropic Sobolev embedding (see
\cite{fargala}).

\begin{proposition}
\label{propositione 1}If condition (\ref{cond immersione}) holds, then for
$q\in\lbrack1,p_{\infty}]$ there is a continuous embedding $W_{0}%
^{1,\overrightarrow{p}}(\Omega)\subset L^{q}(\Omega)$. For $q<p_{\infty}$ the
embedding is compact.
\end{proposition}

\subsection{Assumptions and Definitions}

We consider the following class of nonlinear anisotropic elliptic homogeneous
Dirichlet problems%

\begin{equation}
\left\{
\begin{array}
[c]{lll}%
-\partial_{x_{i}}a_{i}(x,u,\nabla u)+\overset{N}{\underset{i=1}{\sum}}%
H_{i}(x,\nabla u)=f-\partial_{x_{i}}g_{i} & \text{in} & \Omega\\
u=0 & \text{on} & \partial\Omega,
\end{array}
\right.  \label{Problema}%
\end{equation}
where $\Omega$ is a bounded open subset of $\mathbb{R}^{N}$ with Lipschitz
continuous boundary $\partial\Omega$, $N\geq2$, $1<p_{1},\ldots,p_{N}<\infty$
and (\ref{cond immersione}) holds.

\noindent We assume that $a_{i}:\Omega\times\mathbb{R\times R}^{N}%
\rightarrow\mathbb{R}$ and $H_{i}:\Omega\times\mathbb{R}^{N}\rightarrow
\mathbb{R}$ are Carath\'{e}odory functions such that%
\begin{equation}
\overset{N}{\underset{i=1}{\sum}}a_{i}\left(  x,s,\xi\right)  \xi_{i}%
\geq\lambda\overset{N}{\underset{i=1}{\sum}}\left\vert \xi_{i}\right\vert
^{p_{i}}\quad\forall s\in\mathbb{R},\xi\in\mathbb{R}^{N}\text{ and a.e. in
$\Omega,$} \label{ellittic}%
\end{equation}%
\begin{equation}
\left\vert a_{i}\left(  x,s,\xi\right)  \right\vert \leq\gamma\left[
\left\vert s\right\vert ^{\frac{p_{\infty}}{p_{i}^{\prime}}}+\left\vert
\xi_{i}\right\vert ^{p_{i}-1}\right]  , \label{crescita}%
\end{equation}%
\begin{equation}
\left(  a_{i}\left(  x,s,\xi\right)  -a_{i}\left(  x,s,\xi^{\prime}\right)
\right)  \left(  \xi_{i}-\xi_{i}^{\prime}\right)  >0\text{ \ for }\xi_{i}%
\neq\xi_{i}^{\prime}, \label{monotonia}%
\end{equation}%
\begin{equation}
\left\vert H_{i}\left(  x,\xi\right)  \right\vert \leq b_{i}\left\vert \xi
_{i}\right\vert ^{p_{i}-1} \label{crescita H}%
\end{equation}
\noindent where $b_{i},\lambda,\gamma$ are some positive constants for
$i=1,..,N$.

\noindent Moreover we suppose that%
\begin{equation}
f\in L^{p_{\infty}^{\prime}}(\Omega) \label{ff}%
\end{equation}
and
\begin{equation}
g_{i}\in L^{p_{i}^{\prime}}(\Omega)\text{ \ for }i=1,..,N.\quad\label{g}%
\end{equation}

\noindent We observe that in (\ref{crescita H}) we can also assume that
$b_{i}\in L^{r_{i}}(\Omega)$ with $\frac{1}{r_{i}}=\frac{1}{p_{\infty}%
^{\prime}}-\frac{1}{p_{i}^{\prime}}$ for $i=1,..,N$.

\noindent Finally we recall the definition of weak solution. A function $u\in
W_{0}^{1,\overrightarrow{p}}(\Omega)$ is a weak solution to Problem
(\ref{Problema}) if
\[
\overset{N}{\underset{i=1}{\sum}}\int_{\Omega}\left[  a_{i}(x,u,\nabla
u)\varphi_{x_{i}}+H_{i}(x,\nabla u)\varphi\right]  =\int_{\Omega}\left[
f\varphi+\overset{N}{\underset{i=1}{\sum}}g_{i}\varphi_{x_{i}}\right]  ,\text{
\ \ }\forall\varphi\in W_{0}^{1,\overrightarrow{p}}(\Omega).
\]

\section{Strongly monotone operators}

In this section we consider Problem (\ref{Problema}) with $H_{i}\equiv0$ under
the assumptions of strongly monotonicity of the operator and Lipschitz
continuity of $a_{i}$. More precisely we study the following class of
nonlinear anisotropic elliptic homogeneous Dirichlet problems%

\begin{equation}
\left\{
\begin{array}
[c]{lll}%
-\partial_{x_{i}}a_{i}(x,u,\nabla u)=f-\partial_{x_{i}}g_{i} & \text{in} &
\Omega\\
u=0 & \text{on} & \partial\Omega.
\end{array}
\right.  \label{P0}%
\end{equation}
Let us assume that (\ref{ellittic})-(\ref{monotonia}), (\ref{ff}) and
(\ref{g}) hold and that functions $a_{i}$\ satisfy
\begin{equation}
\left(  a_{i}\left(  x,s,\xi\right)  -a_{i}\left(  x,s,\xi^{\prime}\right)
\right)  \left(  \xi_{i}-\xi_{i}^{\prime}\right)  \geq\alpha\left(
\varepsilon+\left\vert \xi_{i}\right\vert +\left\vert \xi_{i}^{\prime
}\right\vert \right)  ^{p_{i}-2}\left\vert \xi_{i}-\xi_{i}^{\prime}\right\vert
^{2} \label{monotonia forte}%
\end{equation}
with $\alpha>0$ and $\varepsilon\geq0$ and the following Lipschitz condition%
\begin{equation}
\left\vert a_{i}\left(  x,s,\xi\right)  -a_{i}\left(  x,s^{\prime},\xi\right)
\right\vert \leq\beta\left(  \theta+\left\vert \xi_{i}\right\vert ^{p_{i}%
-1}\right)  \left\vert s-s^{\prime}\right\vert \label{lip a}%
\end{equation}
with $\beta>0,\theta\geq0$ for $i=1,..,N.$

\noindent By the classical Leray-Lions theorem (see \cite{lions}) there exists
at least a weak solution (see also \cite{li}) to Problem (\ref{P0}) as in the
isotropic case.

As far as the uniqueness is concerned, we will investigate separately the case
when at least one $p_{i}\leq2$ and the case when every $p_{i}>2$ for
$i=1,..,N.$\ In this last case as for $p-$Laplace with $p>2$, there is no
uniqueness in general (see the counterexample in \cite{alvino}). Then
assumption (\ref{monotonia forte}) with $\varepsilon>0$ seems to be necessary
to get a uniqueness result.

\begin{theorem}
\label{th p>2}Let us assume $p_{i}>2$ for $i=1,..,N,$ (\ref{cond immersione}),
(\ref{ellittic}), (\ref{crescita}),(\ref{ff}), (\ref{g}),
(\ref{monotonia forte}) with $\varepsilon>0$ and (\ref{lip a})$.$ Then there
exists a unique weak solution to Problem (\ref{P0}).
\end{theorem}

\begin{proof}
Let $u$ and $v$ be two weak solutions to Problem (\ref{P0}). Let us denote
$w=\left(  u-v\right)  ^{+},D=\left\{  x\in\Omega:w>0\right\}  ,D_{t}=\left\{
x\in D:w<t\right\}  $ for $t\in\left[  0,\sup w\right[  $ and $T_{t}$ the
truncation function at height $t$. Suppose that $D$ has positive
measure.\ Using $\varphi=\frac{T_{t}(w)}{t}$ as test function in the
difference of the equations, we obtain
\[
\overset{N}{\underset{i=1}{\sum}}\int_{D_{t}}\left[  a_{i}\left(  x,u,\nabla
u\right)  -a_{i}\left(  x,v,\nabla v\right)  \right]  \partial_{x_{i}}%
\varphi\leq0.
\]
For the convenience of the reader we are explicitly writing the sum sign. By
(\ref{monotonia forte}) and (\ref{lip a}) we get
\begin{equation}
\overset{N}{\underset{i=1}{\sum}}\int_{D_{t}}\left(  \varepsilon+\left\vert
\partial_{x_{i}}u\right\vert +\left\vert \partial_{x_{i}}v\right\vert \right)
^{p_{i}-2}\left\vert \partial_{x_{i}}\varphi\right\vert ^{2}\leq\frac{\beta
}{\alpha}\overset{N}{\underset{i=1}{\sum}}\int_{D_{t}}\left(  \theta
+\left\vert \partial_{x_{i}}v\right\vert ^{p_{i}-1}\right)  \left\vert
\partial_{x_{i}}\varphi\right\vert . \label{eq1}%
\end{equation}
Using Young inequality with some $\delta>0$ we have%
\begin{align*}
&  \int_{D_{t}}\left(  \theta+\left\vert \partial_{x_{i}}v\right\vert
^{p_{i}-1}\right)  \left\vert \partial_{x_{i}}\varphi\right\vert \\
&  \leq\theta\left(  \frac{\delta}{2}\int_{D_{t}}\left\vert \partial_{x_{i}%
}\varphi\right\vert ^{2}+\frac{1}{4\delta}\left\vert D_{t}\right\vert \right)
+\frac{\delta}{2}\int_{D_{t}}\left(  \left\vert \partial_{x_{i}}%
\varphi\right\vert ^{2}\left\vert \partial_{x_{i}}v\right\vert ^{p_{i}%
-2}\right)  +\frac{1}{4\delta}\int_{D_{t}}\left\vert \partial_{x_{i}%
}v\right\vert ^{p_{i}},
\end{align*}
then by (\ref{eq1}), choosing $\delta$ small enough, we obtain
\begin{equation}
\overset{N}{\underset{i=1}{\sum}}\int_{D_{t}}\left(  \varepsilon+\left\vert
\partial_{x_{i}}u\right\vert +\left\vert \partial_{x_{i}}v\right\vert \right)
^{p_{i}-2}\left\vert \partial_{x_{i}}\varphi\right\vert ^{2}\leq c\left(
N\left\vert D_{t}\right\vert +\overset{N}{\underset{i=1}{\sum}}\int_{D_{t}%
}\left\vert \partial_{x_{i}}v\right\vert ^{p_{i}}\right)  \label{mm}%
\end{equation}
for some positive constant $c$ independent on $t$.$\ $By Young inequality,
(\ref{sobolev}) and (\ref{mm}) we obtain
\begin{align*}
NC_{S}\left\vert D\backslash D_{t}\right\vert ^{1-\frac{1}{N}}  &  \leq
NC_{S}\left\Vert \varphi\right\Vert _{L^{\frac{N}{N-1}}\left(  D\right)  }\\
&  \leq\overset{N}{\underset{i=1}{\sum}}\int_{D_{t}}\left\vert \partial
_{x_{i}}\varphi\right\vert \leq\frac{N}{2}\left\vert D_{t}\right\vert
+\frac{1}{2}\overset{N}{\underset{i=1}{\sum}}\int_{D_{t}}\left\vert
\partial_{x_{i}}\varphi\right\vert ^{2}\\
&  \leq\frac{N}{2}\left\vert D_{t}\right\vert +\frac{1}{2}\overset
{N}{\underset{i=1}{\sum}}\frac{1}{\varepsilon^{p_{i}-2}}\int_{D_{t}}\left\vert
\partial_{x_{i}}\varphi\right\vert ^{2}\left(  \varepsilon+\left\vert
\partial_{x_{i}}u\right\vert +\left\vert \partial_{x_{i}}v\right\vert \right)
^{p_{i}-2}\\
&  \leq\frac{N}{2}\left\vert D_{t}\right\vert +\frac{1}{2}c\left(  N\left\vert
D_{t}\right\vert +\overset{N}{\underset{i=1}{\sum}}\int_{D_{t}}\left\vert
\partial_{x_{i}}v\right\vert ^{p_{i}}\right)  .
\end{align*}
The last term tends to zero when $t$ goes to zero; this implies
\begin{equation}
\left\vert D\right\vert =\underset{t\rightarrow0}{\lim}\left\vert D\backslash
D_{t}\right\vert =0, \label{lim D}%
\end{equation}
from which the conclusion follows.
\end{proof}

\bigskip

This approach also works if we replace hypothesis (\ref{monotonia forte}) by
\begin{equation}
\left(  a_{i}\left(  x,s,\xi\right)  -a_{i}\left(  x,s,\xi^{\prime}\right)
\right)  \left(  \xi_{i}-\xi_{i}^{\prime}\right)  \geq\alpha\left\vert \xi
_{i}-\xi_{i}^{\prime}\right\vert ^{p_{i}}\text{ \ \ }\alpha>0
\label{diff a pi}%
\end{equation}
and hypothesis (\ref{lip a}) by%
\begin{equation}
\left\vert a_{i}\left(  x,s,\xi\right)  -a_{i}\left(  x,s^{\prime},\xi\right)
\right\vert \leq\beta\left(  \theta+\left\vert \xi_{i}\right\vert ^{p_{i}%
-1}+\left(  \left\vert s\right\vert +\left\vert s^{\prime}\right\vert \right)
^{q_{i}}\right)  \left\vert s-s^{\prime}\right\vert \label{lip omega}%
\end{equation}
\ for some $q_{i}>0$, $\beta>0,\theta>0$ for $i=1,..,N.$

\begin{theorem}
\label{th p>2 bis}Let us assume $p_{i}>2$ for $i=1,..,N,$
(\ref{cond immersione}), (\ref{ellittic}), (\ref{crescita}),(\ref{ff}),
(\ref{g}), (\ref{diff a pi}) and (\ref{lip omega}) with $0<q_{i}\leq
\frac{p_{\infty}}{p_{i}}$ Then there exists a unique weak solution to Problem
(\ref{P0}).
\end{theorem}

\begin{proof}
We argue as in the proof of Theorem \ref{th p>2} taking into account the
following extra term in (\ref{eq1})%
\[
\overset{N}{\underset{i=1}{\sum}}\int_{D_{t}}\left(  \left\vert u\right\vert
+\left\vert v\right\vert \right)  ^{q_{i}}\left\vert \partial_{x_{i}}%
\varphi\right\vert .
\]
Using (\ref{diff a pi})\ and (\ref{lip omega}) instead of
(\ref{monotonia forte}) and (\ref{lip a}) respectively, we obtain%
\[
\alpha\overset{N}{\underset{i=1}{\sum}}\int_{D_{t}}\left\vert \partial_{x_{i}%
}\varphi\right\vert ^{p_{i}}\leq\beta\overset{N}{\underset{i=1}{\sum}}%
\int_{D_{t}}\left(  \theta+\left\vert \partial_{x_{i}}v\right\vert ^{p_{i}%
-1}\right)  \left\vert \partial_{x_{i}}\varphi\right\vert +\overset
{N}{\underset{i=1}{\sum}}\int_{D_{t}}\left(  \left\vert u\right\vert
+\left\vert v\right\vert \right)  ^{q_{i}}\left\vert \partial_{x_{i}}%
\varphi\right\vert .
\]
Using Young inequality with some $\delta>0$ and choosing $\delta$ small
enough, we obtain the analogue of (\ref{mm})%
\begin{equation}
\overset{N}{\underset{i=1}{\sum}}\int_{D_{t}}\left\vert \partial_{x_{i}%
}\varphi\right\vert ^{p_{i}}\leq c\left(  N\left\vert D_{t}\right\vert
+\overset{N}{\underset{i=1}{\sum}}\int_{D_{t}}\left\vert \partial_{x_{i}%
}v\right\vert ^{p_{i}}+\overset{N}{\underset{i=1}{\sum}}\int_{D_{t}}\left(
\left\vert u\right\vert +\left\vert v\right\vert \right)  ^{p_{i}q_{i}%
}\right)  , \label{mmm2}%
\end{equation}
for some positive constant $c$ independent on $t$.$\ $Since $0<q_{i}\leq
\frac{p_{\infty}}{p_{i}},$ (\ref{mmm2}) allows us to conclude.
\end{proof}

\bigskip

\begin{remark}
\label{remarc omega}If in Theorem \ref{th p>2 bis} we assume
(\ref{monotonia forte}) holds instead of (\ref{diff a pi}), we can take
$0<q_{i}\leq\frac{p_{\infty}}{2}.$ Moreover Theorem \ref{th p>2 bis} holds if
we replace (\ref{lip omega}) by
\[
\left\vert a_{i}\left(  x,s,\xi\right)  -a_{i}\left(  x,s^{\prime},\xi\right)
\right\vert \leq\beta\left(  \theta+\left\vert \xi_{i}\right\vert ^{p_{i}%
-1}+\left(  \left\vert s\right\vert +\left\vert s^{\prime}\right\vert \right)
^{q_{i}}\right)  \omega(\left\vert s-s^{\prime}\right\vert )
\]
\ for $i=1,..,N,$ where $\omega:\left[  0,+\infty\right[  \rightarrow\left[
0,+\infty\right[  $ is such that $\omega(s)\leq s$ for $0\leq s\leq\rho$ for
some $\rho>0.$
\end{remark}

\bigskip

Now we study Problem (\ref{P0}) when at least one $p_{i}$ is less or equal to
$2$ and $\varepsilon=0$ in (\ref{monotonia forte}). We argue as in Theorem
\ref{th p>2} by using Poincar\'{e} inequality (\ref{dis poincare}) instead of
inequality (\ref{sobolev}). The following result is obtained by a different
proof in \cite{chipot} (see Theorem 2.1).

\begin{theorem}
\label{th chipot}Let us assume (\ref{cond immersione}), (\ref{ellittic}),
(\ref{crescita}),(\ref{ff}), (\ref{g}), (\ref{monotonia forte}) with
$\varepsilon=0$ and (\ref{lip a}) with $\theta=0$. If at least one $p_{i}$ is
less or equal to $2$, then there exists a unique weak solution to Problem
(\ref{P0}).
\end{theorem}

\begin{proof}
Arguing as in the proof of Theorem \ref{th p>2} we have%

\begin{equation}
\overset{N}{\alpha\underset{i=1}{\sum}}\int_{D_{t}}\left\vert \partial_{x_{i}%
}\varphi\right\vert ^{2}\left(  \left\vert \partial_{x_{i}}u\right\vert
+\left\vert \partial_{x_{i}}v\right\vert \right)  ^{p_{i}-2}\leq\beta
\overset{N}{\underset{i=1}{\sum}}\int_{D_{t}}\left\vert \partial_{x_{i}%
}v\right\vert ^{p_{i}-1}\left\vert \partial_{x_{i}}\varphi\right\vert .
\label{eq2}%
\end{equation}
By Young inequality with some $\delta>0$ we get%
\begin{equation}
\!\int_{\!D_{t}}\!\!\!\left\vert \partial_{x_{i}}v\right\vert ^{p_{i}%
-1}\!\!\left\vert \partial_{x_{i}}\varphi\right\vert \!\leq\!\!\frac{\delta
}{2}\!\int_{\!D_{t}}\!\!\left\vert \partial_{x_{i}}\varphi\right\vert
^{2}\left(  \!\left\vert \partial_{x_{i}}u\right\vert +\left\vert
\partial_{x_{i}}v\right\vert \!\right)  ^{p_{i}-2}\!+\!\frac{1}{4\delta}%
\!\int_{\!D_{t}}\!\!\left(  \!\left\vert \partial_{x_{i}}v\right\vert
+\left\vert \partial_{x_{i}}v\right\vert \!\right)  ^{p_{i}}\!.\!
\label{intermedia}%
\end{equation}
Putting (\ref{intermedia}) in (\ref{eq2}) and choosing $\delta$ small enough
we obtain
\begin{equation}
\overset{N}{\underset{i=1}{\sum}}\int_{D_{t}}\left\vert \partial_{x_{i}%
}\varphi\right\vert ^{2}\left(  \!\left\vert \partial_{x_{i}}u\right\vert
+\left\vert \partial_{x_{i}}v\right\vert \!\right)  ^{p_{i}-2}\leq
c_{1}\overset{N}{\underset{i=1}{\sum}}\int_{D_{t}}\left(  \left\vert
\partial_{x_{i}}v\right\vert +\left\vert \partial_{x_{i}}v\right\vert \right)
^{p_{i}} \label{eq3333}%
\end{equation}
for some positive constant $c_{1}$ independent on $t.$ Let $p_{j}\leq2.$ Using
Poincar\'{e} inequality (\ref{dis poincare}), Young inequality and
(\ref{eq3333}) we get
\begin{align*}
C_{P}\left\vert D\backslash D_{t}\right\vert  &  \leq\frac{c_{2}}{2}\left[
\int_{D_{t}}\frac{\left\vert \partial_{x_{j}}\varphi\right\vert ^{2}}{\left(
\left\vert \partial_{x_{j}}u\right\vert +\left\vert \partial_{x_{j}%
}v\right\vert \right)  ^{2-p_{j}}}+\int_{D_{t}}\left(  \left\vert
\partial_{x_{j}}v\right\vert +\left\vert \partial_{x_{j}}v\right\vert \right)
^{2-p_{j}}\right] \\
&  \leq\frac{c_{1}c_{2}}{2}\overset{N}{\underset{i=1}{\sum}}\int_{D_{t}%
}\left(  \left\vert \partial_{x_{i}}u\right\vert +\left\vert \partial_{x_{i}%
}v\right\vert \right)  ^{p_{i}}+\frac{c_{2}}{2}\int_{D_{t}}\left(  \left\vert
\partial_{x_{j}}u\right\vert +\left\vert \partial_{x_{j}}v\right\vert \right)
^{2-p_{j}}%
\end{align*}
for some positive constant $c_{2}$ independent on $t$. As in Theorem
\ref{th p>2}, condition (\ref{lim D}) follows, then the assert holds.
\end{proof}

\bigskip

\begin{remark}
\label{remark b}Theorem \ref{th chipot} also holds if we replace (\ref{lip a})
by%
\[
\left\vert a_{i}\left(  x,s,\xi\right)  -a_{i}\left(  x,s^{\prime},\xi\right)
\right\vert \leq\beta\left(  \left\vert \xi_{i}\right\vert ^{p_{i}-1}\right)
\omega(\left\vert s-s^{\prime}\right\vert )\text{ \ \ \ }\beta>0
\]
\ for $i=1,..,N,$ where is like in Remark \ref{remarc omega}. Moreover if we
suppose that every $p_{i}<2$ Theorem \ref{th chipot} holds with $\theta=0$ in
(\ref{lip a}). Finally we stress that Theorems \ref{th p>2}, \ref{th p>2 bis}
and \ref{th chipot} hold if in Problem (\ref{P0}) we add the term $c(x,u)$
with suitable hypotheses (for example $c$ is an increasing function with
respect to $u).$
\end{remark}

\section{Operators with a first order term}

In this section we consider Problem (\ref{Problema}) when the functions
$a_{i}$\ do not depend on $u.$ More precisely under the assumptions
(\ref{ellittic})-(\ref{g}) we consider the following class of nonlinear
anisotropic homogeneous Dirichlet problems:%

\begin{equation}
\left\{
\begin{array}
[c]{lll}%
-\partial_{x_{i}}a_{i}(x,\nabla u)+\overset{N}{\underset{i=1}{\sum}}%
H_{i}(x,\nabla u)=f-\partial_{x_{i}}g & \text{in} & \Omega\\
u=0 & \text{on} & \partial\Omega.
\end{array}
\right.  \label{P1}%
\end{equation}

\subsection{An existence result for Problem (\ref{P1}).}

In this section we prove the existence of at least a weak solution to Problem
(\ref{P1}). To our knowledge this result could not be found in literature.

\noindent The coercivity of the operator is guaranteed only if the norms of
$b_{i}$ are small enough. As usual we consider the approximate problems. Let
$H_{n}^{i}(x,\nabla u)$ be the truncation at levels $\pm n$ of $H_{i}$. It is
well known (see e.g. \cite{lions}) that there exists a weak solution $u_{n}\in
W_{0}^{1,\overrightarrow{p}}(\Omega)$ to problem%
\begin{equation}
\left\{
\begin{array}
[c]{ccc}%
-\partial_{x_{i}}a_{i}(x,\nabla u_{n})+\overset{N}{\underset{i=1}{\sum}}%
H_{n}^{i}(x,\nabla u_{n})=f-\partial_{x_{i}}g_{i} & \text{in} & \Omega\\
u=0 & \text{on} & \partial\Omega.
\end{array}
\right.  \label{approssimato}%
\end{equation}
The first and crucial step is an a priori estimate of $u_{n}.$ For the
convenience of the reader we are writing explicitly the sum sign.

\begin{lemma}
\label{lemma stim} Assume that (\ref{cond immersione}), (\ref{ellittic}%
)-(\ref{crescita H}), (\ref{ff}) and (\ref{g}) hold and let $u_{n}\in
W_{0}^{1,\overrightarrow{p}}(\Omega)$ be a solution to Problem
(\ref{approssimato}). Then we have
\begin{equation}
\overset{N}{\underset{i=1}{\sum}}\int_{\Omega}\left\vert \partial_{x_{i}}%
u_{n}\right\vert ^{p_{i}}\leq C, \label{stima 0}%
\end{equation}
for some positive constant $C$ depending on $N,\Omega,\lambda,\gamma
,p_{i},b_{i},\left\Vert f\right\Vert _{L^{p_{\infty}}(\Omega)},\left\Vert
g_{i}\right\Vert _{L^{p_{i}^{\prime}}(\Omega)}$ for $i=1,..,N.$
\end{lemma}

\begin{proof}
In what follows we do not explicitly write the dependence on $n$. The
technique developed in \cite{bottaro} allows us to avoid the assumption on
smallness of $\left\Vert b_{i}\right\Vert _{L^{\infty}(\Omega)}$. Let $A$ be a
positive real number, that will be chosen later$.$ Then there exists $t$
measurable subsets $\Omega_{1},...,\Omega_{t}$ of $\Omega$ and $t$ functions
$u_{1},...,u_{t}$ such that $\Omega_{i} \cap\Omega_{j} = \emptyset$ for $i\neq
j$,$\left\vert \Omega_{t}\right\vert \leq A$ and $\left\vert \Omega
_{s}\right\vert =A$ for $s\in\left\{  1,..,t-1\right\}  ,\left\{  x\in
\Omega:\left\vert \nabla u_{s}\right\vert \neq0\right\}  \subset\Omega
_{s},\nabla u=\nabla u_{s}$ \ a.e. in $\Omega_{s},\nabla\left(  u_{1}%
+...+u_{s}\right)  u_{s}=\left(  \nabla u\right)  u_{s},u_{1}+...+u_{s}=u$ in
$\Omega$ and\ sign$(u)=$sign$(u_{s})$ \ if $u_{s}\neq0$\ \ for $s\in\left\{
1,..,t\right\}  .$

\noindent Let us fix $s\in\left\{  1,..,t\right\}  $ and let us use $u_{s}$ as
test function in Problem (\ref{approssimato}). Using (\ref{ellittic}), Young
and H\"{o}lder inequalities and Proposition \ref{propositione 1} we obtain%
\begin{equation}
\!\underset{i=1}{\overset{N}{%
{\displaystyle\sum}
}}\!\int_{\!\Omega}\!\left\vert \partial_{x_{i}}u_{s}\right\vert ^{p_{i}%
}\!\!\leq\!\!c_{1}\!\left(  \!\left\Vert f\right\Vert _{L^{p_{\infty}^{\prime
}}}\!d_{s}^{\frac{1}{N}}\!+\!\underset{i=1}{\overset{N}{%
{\displaystyle\sum}
}}\!\int_{\!\Omega}\!\left\vert H^{i}(x,\nabla u)\right\vert \left\vert
u_{s}\right\vert \!+\!\underset{i=1}{\overset{N}{%
{\displaystyle\sum}
}}\!\left\Vert g_{i}\right\Vert _{L^{p_{i}^{\prime}}}^{p_{i}^{\prime}%
}\!\!\right)  \! \label{equazione 1}%
\end{equation}
for some constant $c_{1}>0,$ where $d_{s}=\underset{i}{%
{\displaystyle\prod}
}\left(  \int_{\Omega}\left\vert \partial_{x_{i}}u_{s}\right\vert ^{p_{i}%
}\right)  ^{\frac{1}{p_{i}}}.$ Here and in what follows the constants depend
on the data but not on the function $u.$

\noindent Using condition (\ref{crescita H}),\ H\"{o}lder and Young
inequalities and Proposition \ref{propositione 1} we get
\begin{align}
\left\vert \underset{i=1}{\overset{N}{%
{\displaystyle\sum}
}}\int_{\Omega}H^{i}(x,\nabla u)u_{s}\right\vert  &  \leq\underset
{i=1}{\overset{N}{%
{\displaystyle\sum}
}}b_{i}\int_{\Omega}\left\vert \partial_{x_{i}}u\right\vert ^{p_{i}%
-1}\left\vert u_{s}\right\vert \label{b2}\\
\!\!  &  \leq\underset{i=1}{\overset{N}{%
{\displaystyle\sum}
}}b_{i}\underset{\sigma=1}{\overset{s}{%
{\displaystyle\sum}
}}\int_{\Omega_{\sigma}}\left\vert \partial_{x_{i}}u_{\sigma}\right\vert
^{p_{i}-1}\left\vert u_{s}\right\vert \nonumber\\
\!\!  &  \leq\!c_{2}\!\underset{i=1}{\overset{N}{%
{\displaystyle\sum}
}}\!\underset{\sigma=1}{\overset{s}{%
{\displaystyle\sum}
}}\!\left[  \!\int_{\!\Omega_{\sigma}}\!\left\vert \partial_{x_{i}}u_{\sigma
}\right\vert ^{p_{i}}\left\vert \Omega_{\sigma}\right\vert ^{\frac{1}{p_{i}%
}-\frac{1}{p_{\infty}}}\!+\!d_{s}^{\frac{p_{i}}{N}}\!\left\vert \Omega
_{\sigma}\right\vert ^{\frac{1}{p_{i}}-\frac{1}{p_{\infty}}}\!\right]
\!\nonumber\\
\!\!  &  \leq\!c_{2}\!\overset{N}{\underset{i=1}{\sum}}\!A^{\frac{1}{p_{i}%
}-\frac{1}{p_{\infty}}}\!\left[  \!\int_{\!\Omega_{s}}\!\left\vert
\partial_{x_{i}}u_{s}\right\vert ^{pi}\!+\!\underset{\sigma=1}{\overset{s-1}{%
{\displaystyle\sum}
}}\!\int_{\!\Omega_{\sigma}}\!\left\vert \partial_{x_{i}}u_{\sigma}\right\vert
^{p_{i}}\!+\!d_{s}^{\frac{p_{i}}{N}}\!\right]  \!\nonumber
\end{align}

\noindent for some constant $c_{2}>0.$ Putting (\ref{b2}) in
(\ref{equazione 1}) we obtain
\begin{align}
&  \underset{i=1}{\overset{N}{%
{\displaystyle\sum}
}}\int_{\Omega}\left\vert \partial_{x_{i}}u_{s}\right\vert ^{p_{i}}\leq
c_{1}\left\{  \left\Vert f\right\Vert _{L^{p_{\infty}^{\prime}}}d_{s}%
^{\frac{1}{N}}+\underset{i=1}{\overset{N}{%
{\displaystyle\sum}
}}\left\Vert g_{i}\right\Vert _{L^{p_{i}^{\prime}}}^{p_{i}^{\prime}}+\right.
\label{eqqqqqqqqq}\\
&  \left.  c_{2}\overset{N}{\underset{i=1}{\sum}}A^{\frac{1}{p_{i}}-\frac
{1}{p_{\infty}}}\left[  \int_{\Omega_{s}}\left\vert \partial_{x_{i}}%
u_{s}\right\vert ^{pi}+\underset{\sigma=1}{\overset{s-1}{%
{\displaystyle\sum}
}}\int_{\Omega_{\sigma}}\left\vert \partial_{x_{i}}u_{\sigma}\right\vert
^{p_{i}}+\overset{N}{\underset{i=1}{\sum}}d_{s}^{\frac{p_{i}}{N}}\right]
\right\}  .\nonumber
\end{align}
If $A$ is such that
\begin{equation}
1-c_{1}c_{2}\overset{N}{\underset{i=1}{\sum}}A^{\frac{1}{p_{i}}-\frac
{1}{p_{\infty}}}>0, \label{cond 1}%
\end{equation}
inequality (\ref{eqqqqqqqqq}) becomes
\begin{align}
&  \underset{i=1}{\overset{N}{%
{\displaystyle\sum}
}}\int_{\Omega}\left\vert \partial_{x_{i}}u_{s}\right\vert ^{p_{i}}\leq
c_{3}\left\{  \left\Vert f\right\Vert _{L^{p_{\infty}^{\prime}}}d_{s}%
^{\frac{1}{N}}+\underset{i=1}{\overset{N}{%
{\displaystyle\sum}
}}\left\Vert g_{i}\right\Vert _{L^{p_{i}^{\prime}}}^{p_{i}^{\prime}}+\right.
\label{equazione s 2}\\
&  \left.  \underset{\sigma=1}{\overset{s-1}{%
{\displaystyle\sum}
}}\underset{i=1}{\overset{N}{%
{\displaystyle\sum}
}}A^{\frac{1}{p_{i}}-\frac{1}{p_{\infty}}}\left(  \overset{N}{\underset
{j=1}{\sum}}\int_{\Omega_{\sigma}}\left\vert \partial_{x_{j}}u_{\sigma
}\right\vert ^{p_{j}}\right)  +\overset{N}{\underset{i=1}{\sum}}A^{\frac
{1}{p_{i}}-\frac{1}{p_{\infty}}}d_{s}^{\frac{p_{i}}{N}}\right\} \nonumber
\end{align}
for some constant $c_{3}>0.$ For $s=1$ we get \
\begin{equation}
\!\!\!\!\int_{\!\Omega}\!\left\vert \partial_{x_{i}}u_{1}\right\vert ^{p_{i}%
}\!\leq\!\underset{i=1}{\overset{N}{%
{\displaystyle\sum}
}}\!\int_{\!\Omega}\!\left\vert \partial_{x_{i}}u_{1}\right\vert ^{p_{i}%
}\!\leq\!\!c_{3}\!\left[  \!\left\Vert f\right\Vert _{L^{p_{\infty}^{\prime}}%
}\!\!d_{1}^{\frac{1}{N}}\!+\!\underset{i=1}{\overset{N}{%
{\displaystyle\sum}
}}\!\left\Vert g_{i}\right\Vert _{L^{p_{i}^{\prime}}}^{p_{i}^{\prime}%
}\!\!+\!\overset{N}{\underset{i=1}{\sum}}A^{\frac{1}{p_{i}}-\frac{1}%
{p_{\infty}}}\!d_{1}^{\frac{p_{i}}{N}}\!\!\right]  \!.\! \label{s01 bis}%
\end{equation}
Let us choose $A$ such that (\ref{cond 1}) and%
\[
1-c_{3}\overset{N}{\underset{i=1}{\sum}}A^{\frac{N}{p_{i}}\left(  \frac
{1}{p_{i}}-\frac{1}{p_{\infty}}\right)  }>0
\]
hold. For example we can take $A<\min\left\{  1,\left(  \frac{1}{2c_{3}%
}\right)  ^{\frac{1}{\underset{i=1,..N}{\min}\left\{  \frac{1}{p_{i}}-\frac
{1}{p_{\infty}}\right\}  }},\left(  \frac{1}{2c_{3}}\right)  ^{\frac
{1}{\underset{i=1,..N}{\min}\left\{  \frac{N}{p_{i}}\left(  \frac{1}{p_{i}%
}-\frac{1}{p_{\infty}}\right)  \right\}  }}\right\}  .$ By this choice we
obtain
\[
d_{1}=\underset{i}{%
{\displaystyle\prod}
}\left(  \int_{\Omega}\left\vert \partial_{x_{i}}u_{1}\right\vert ^{p_{i}%
}\right)  ^{\frac{1}{p_{i}}}\leq c_{4}\left[  \left(  \left\Vert f\right\Vert
_{L^{p_{\infty}^{\prime}}}^{\frac{N}{\overline{p}}}+\left\Vert \nu
_{2}\right\Vert _{L^{p_{\infty}^{\prime}}}^{\frac{N}{\overline{p}}}\right)
d_{1}^{\frac{1}{\overline{p}}}+\underset{i=1}{\overset{N}{%
{\displaystyle\sum}
}}\left\Vert g_{i}\right\Vert _{L^{p_{i}^{\prime}}}^{p_{i}^{\prime}}\right]
.
\]
Then there exists a constant $c_{5}>0$ such that $d_{1}\leq c_{5}$ and by
(\ref{s01 bis}) we obtain
\begin{equation}
\underset{i=1}{\overset{N}{%
{\displaystyle\sum}
}}\int_{\Omega}\left\vert \partial_{x_{i}}u_{1}\right\vert ^{p_{i}}\leq c_{6}
\label{C1}%
\end{equation}
for some constant $c_{6}>0.$ Moreover using (\ref{C1}) in (\ref{equazione s 2}%
) and iterating on $s$ we have
\[
\underset{i=1}{\overset{N}{%
{\displaystyle\sum}
}}\int_{\Omega}\left\vert \partial_{x_{i}}u_{s}\right\vert ^{p_{i}}\leq
c_{7}\left[  \left\Vert f\right\Vert _{L^{p_{\infty}^{\prime}}}d_{s}^{\frac
{1}{N}}+\underset{i=1}{\overset{N}{%
{\displaystyle\sum}
}}\left\Vert g_{i}\right\Vert _{L^{p_{i}^{\prime}}}^{p_{i}^{\prime}%
}+1+\overset{N}{\underset{i=1}{\sum}}A^{\frac{1}{p_{i}}-\frac{1}{p_{\infty}}%
}d_{1}^{\frac{p_{i}}{N}}\right]  ,
\]
then arguing as before we obtain
\begin{equation}
\underset{i=1}{\overset{N}{%
{\displaystyle\sum}
}}\int_{\Omega}\left\vert \partial_{x_{i}}u_{s}\right\vert ^{p_{i}}\leq c_{8}
\label{Cs}%
\end{equation}
for some constant $c_{8}>0.$ The assertion follows immediately since
$\left\Vert u\right\Vert _{W_{0}^{1,\overrightarrow{p}}(\Omega)}\leq
k\underset{i=1}{\overset{N}{%
{\displaystyle\sum}
}}$ $\left(  \underset{s=1}{\overset{t}{%
{\displaystyle\sum}
}}%
{\displaystyle\int_{\Omega}}
\left\vert \partial_{x_{i}}u_{s}\right\vert ^{p_{i}}\right)  ^{\frac{1}{p_{i}%
}}$ for some positive $k>0$.
\end{proof}

\bigskip

Now we are able to prove the following existence result.

\begin{theorem}
Assume that (\ref{cond immersione}), (\ref{ellittic})-(\ref{crescita H}),
(\ref{ff}) and (\ref{g}) hold, then there exists at least a weak solution to
Problem (\ref{P1}).
\end{theorem}

\begin{proof}
We give only a sketch of the proof, because it is standard. By (\ref{stima 0})
the sequence $\partial_{x_{i}}u_{n}$ is bounded in $L^{p_{i}}(\Omega)$ so we
have that
\begin{equation}
\partial_{x_{i}}u_{n}\rightharpoonup\partial_{x_{i}}u\text{ weakly in
}L^{p_{i}}(\Omega)\text{ \ for }i=1,...,N, \label{con debole}%
\end{equation}%
\begin{equation}
u_{n}\rightarrow u\text{ strongly in }L^{p_{-}}(\Omega)\text{ with }p_{-}%
=\min\left\{  p_{1},..,p_{N}\right\}  \label{con foret}%
\end{equation}
for some $u$ and for some subsequence, which we still denote by $u_{n}.$ We
can argue as in \cite{boccardo-murat} to prove
\begin{equation}
\partial_{x_{i}}u_{n}\rightarrow\partial_{x_{i}}u\text{ a.e. in }\Omega\text{
\ for }i=1,...,N. \label{a.e gradienti}%
\end{equation}
Using convergence (\ref{a.e gradienti}) we have for $i=1,...,N$%
\[
\left\{
\begin{array}
[c]{c}%
a_{i}(x,\nabla u_{n})\rightarrow a_{i}(x,\nabla u)\text{ a.e. in }\Omega,\\
H_{n}^{i}(x,\nabla u_{n})\rightarrow H_{i}(x,\nabla u)\text{ a.e. in }\Omega.
\end{array}
\right.
\]
Moreover by (\ref{crescita}) and (\ref{crescita H}) for any $q_{i}\in\left[
1,p_{i}^{\prime}\right[  $ we have
\[
\int_{E}\left\vert a_{i}(x,\nabla u_{n})\right\vert ^{q_{i}}\leq c\left[
\left(  \int_{E}\left\vert u(x)\right\vert ^{p_{\infty}}\right)  ^{\frac
{q_{i}}{p_{i}^{\prime}}}+\left(  \int_{E}\left\vert \partial_{x_{i}%
}u(x)\right\vert ^{p_{i}}\right)  ^{\frac{q_{i}}{p_{i}^{\prime}}}\right]
\left\vert E\right\vert ^{1-\frac{q_{i}}{p_{i}^{\prime}}}.
\]
and%
\[
\int_{E}\left\vert H_{n}^{i}(x,\nabla u_{n})\right\vert ^{q_{i}}\leq c\left(
\int_{E}\left\vert \partial_{x_{i}}u(x)\right\vert ^{p_{i}}\right)
^{\frac{q_{i}}{p_{i}^{\prime}}}\left\vert E\right\vert ^{1-\frac{q_{i}}%
{p_{i}^{\prime}}}%
\]
for some positive constant $c$, for $i=1,...,N$ and for any measurable subset
$E.$ Then Vitali Theorem assures
\[
a_{i}(x,\nabla u_{n})\rightarrow a_{i}(x,\nabla u)\text{ and }H_{n}%
^{i}(x,\nabla u_{n})\rightarrow H_{i}(x,\nabla u)\text{ strongly in }L^{q_{i}%
}(\Omega)
\]
for $q_{i}\in\left[  1,p_{i}^{\prime}\right[  ,$ that allow us to pass to the
limit in the approximate problems.
\end{proof}

\begin{remark}
The last theorem still holds if in (\ref{crescita H}) we assume $b_{i}\in
L^{r_{i}}(\Omega)$ with $\frac{1}{r_{i}}=\frac{1}{p_{i}}-\frac{1}{p_{\infty}}
$ for $i=1,..,N$.
\end{remark}

\subsection{Some uniqueness results for Problem (\ref{P1}$)$}

The first uniqueness result is obtained when every $p_{i}\ $is not greater
than $2$ assuming the following Lipschitz condition on $H_{i}$%

\begin{equation}
\left\vert H_{i}(x,\xi)-H_{i}(x,\xi^{\prime})\right\vert \leq h\frac
{\left\vert \xi_{i}-\xi_{i}^{\prime}\right\vert }{\left(  \eta+\left\vert
\xi_{i}\right\vert +\left\vert \xi_{i}^{\prime}\right\vert \right)
^{\sigma_{i}}} \label{lip H}%
\end{equation}
for some constants $h>0$, $\eta>0$ and $\sigma_{i}>0$ for $i=1,...,N.$

\begin{theorem}
\label{Th+H1}Let $1<p_{i}\leq2$ if $N=2$, $\frac{2N}{N+2}\leq p_{i}\leq2$ if
$N\geq3$ and $\sigma_{i}\geq1-\frac{p_{i}}{2}$ for $i=1,...,N.$ Let us assume
(\ref{cond immersione}), (\ref{ellittic})-(\ref{g}), (\ref{monotonia forte})
with $\varepsilon=0$ and (\ref{lip H}) with $\eta>0$. Then there exists a
unique weak solution to Problem (\ref{P1}).
\end{theorem}

\begin{proof}
Let us suppose $u$ and $v$ are two weak solutions to Problem (\ref{P1}) and
denote $w=\left(  u-v\right)  ^{+}$ and $E_{t}=\left\{  x\in\Omega:t<w<\sup
w\right\}  $ for $t\in\left[  0,\sup w\right[  .$ We use
\[
w_{t}=\left\{
\begin{array}
[c]{lll}%
w(x)-t &  & \text{if }w(x)>t\\
0 &  & \text{otherwise}%
\end{array}
\right.
\]
as test function in the difference of the equations. Strong monotonicity
(\ref{monotonia forte}) with $\varepsilon=0$ and the Lipschitz condition
(\ref{lip H}) with $\eta>0$ give%
\[
\overset{N}{\underset{i=1}{\sum}}\int_{E_{t}}\frac{\left\vert \partial_{x_{i}%
}w_{t}\right\vert ^{2}}{\left(  \left\vert \partial_{x_{i}}u\right\vert
+\left\vert \partial_{x_{i}}v\right\vert \right)  ^{2-p_{i}}}\leq\frac
{h}{\alpha}\overset{N}{\underset{i=1}{\sum}}\int_{E_{t}}\frac{\left\vert
\partial_{x_{i}}w_{t}\right\vert w_{t}}{\left(  \eta+\left\vert \partial
_{x_{i}}u\right\vert +\left\vert \partial_{x_{i}}v\right\vert \right)
^{\sigma_{i}}}%
\]
Since $\sigma_{i}\geq1-\frac{p_{i}}{2}$ by Young inequality and some easy
computations we have
\begin{equation}
\overset{N}{\underset{i=1}{\sum}}\int_{E_{t}}\frac{\left\vert \partial_{x_{i}%
}w_{t}\right\vert ^{2}}{\left(  \left\vert \partial_{x_{i}}u\right\vert
+\left\vert \partial_{x_{i}}v\right\vert \right)  ^{2-p_{i}}}\leq c\int
_{E_{t}}w_{t}^{2} \label{stima}%
\end{equation}
for some positive constant $c$ independent on $t.$ Moreover by
(\ref{sobolev 2}) and H\"{o}lder inequality we get%
\begin{align*}
\frac{1}{C_{S}}\left(  \int_{E_{t}}w_{t}^{2}\right)  ^{\frac{1}{2}}  &
\leq\overset{N}{\underset{i=1}{%
{\displaystyle\prod}
}}\left(  \int_{E_{t}}\left\vert \partial_{x_{i}}w_{t}\right\vert ^{\frac
{2N}{N+2}}\right)  ^{\frac{N+2}{2N^{2}}}\\
\!  &  \leq\!\overset{N}{\underset{i=1}{%
{\displaystyle\prod}
}}\!\left(  \!\int_{\!E_{t}}\!\!\frac{\left\vert \partial_{x_{i}}%
w_{t}\right\vert ^{2}}{\left(  \left\vert \partial_{x_{i}}u\right\vert
\!+\!\left\vert \partial_{x_{i}}v\right\vert \right)  ^{2-p_{i}}}\!\right)
^{\!\frac{1}{2N}}\!\!\left(  \!\int_{\!E_{t}}\!\!\!\left(  \left\vert
\partial_{x_{i}}u\right\vert \!+\!\left\vert \partial_{x_{i}}v\right\vert
\right)  ^{\left(  2-p_{i}\right)  \frac{N}{2}}\!\right)  ^{\!\frac{1}{N^{2}}%
}\!,\!
\end{align*}
then by (\ref{geo-arti}) we obtain%
\[
\!\frac{1}{C_{S}^{2}}\!\int_{E_{t}}w_{t}^{2}\!\leq\!N^{2}\!\overset
{N}{\underset{i=1}{\sum}}\!\int_{\!E_{t}}\!\frac{\left\vert \partial_{x_{i}%
}w_{t}\right\vert ^{2}}{\left(  \left\vert \partial_{x_{i}}u\right\vert
+\left\vert \partial_{x_{i}}v\right\vert \right)  ^{2-p_{i}}}\!\overset
{N}{\underset{i=1}{\sum}}\!\left(  \!\int_{\!E_{t}}\!\left(  \left\vert
\partial_{x_{i}}u\right\vert \!+\!\left\vert \partial_{x_{i}}v\right\vert
\!\right)  ^{\!\left(  2-p_{i}\right)  \frac{N}{2}}\!\right)  ^{\!\frac{2}{N}%
}.
\]
Finally using (\ref{stima}) we get%
\[
\frac{1}{C_{S}^{2}}\leq cN^{2}\overset{N}{\underset{i=1}{\sum}}\left(
\int_{E_{t}}\left(  \left\vert \partial_{x_{i}}u\right\vert +\left\vert
\partial_{x_{i}}v\right\vert \right)  ^{\left(  2-p_{i}\right)  \frac{N}{2}%
}\right)  ^{\frac{2}{N}}.
\]
Since $\left(  2-p_{i}\right)  \frac{N}{2}\leq p_{i}$ we have%
\[
\underset{t\rightarrow\sup w}{\lim}\int_{E_{t}}\left(  \left\vert
\partial_{x_{i}}u\right\vert +\left\vert \partial_{x_{i}}v\right\vert \right)
^{\left(  2-p_{i}\right)  \frac{N}{2}}=0,
\]
that gives a contradiction.
\end{proof}

\bigskip

The second result is obtained when every $p_{i}$ is greater than $2$ but
$\varepsilon>0$ in (\ref{monotonia forte}) and we assume the following
Lipschitz condition on $H_{i}$
\begin{equation}
\left\vert H_{i}(x,\xi)-H_{i}(x,\xi^{\prime})\right\vert \leq h_{i}(x)\left(
\left\vert \xi_{i}\right\vert +\left\vert \xi_{i}^{\prime}\right\vert \right)
^{\sigma_{i}}\left\vert \xi_{i}-\xi_{i}^{\prime}\right\vert \label{lip H 2}%
\end{equation}
with $\sigma_{i}\geq0,$ $h_{i}\in L^{s_{i}}(\Omega)$ and $s_{i}\geq
\frac{p_{\infty}p_{i}}{p_{\infty}-p_{i}}.$\bigskip

\begin{theorem}
\label{th+H2} Let us suppose
\[
N\geq3,\text{ }2\leq p_{i}\leq\frac{2Ns_{i}}{Ns_{i}-2s_{i}-2N},
\]%
\[
s_{i}\geq\max\left\{  N,\frac{p_{\infty}p_{i}}{p_{\infty}-p_{i}}\right\}
\text{ and }0\leq\sigma_{i}\leq\frac{p_{i}}{N}-\frac{p_{i}}{s_{i}}+\frac
{p_{i}-2}{2}%
\]
for $i=1,...,N.$ Let us assume (\ref{cond immersione}), (\ref{ellittic}%
)-(\ref{g}), (\ref{monotonia forte}) with $\varepsilon>0$ and (\ref{lip H 2}).
Then there exists a unique weak solution to Problem (\ref{P1}).
\end{theorem}

\begin{proof}
Arguing as in the proof of Theorem \ref{Th+H1} we get%
\begin{equation}
\overset{N}{\underset{i=1}{\sum}}\!\int_{\!E_{t}}\!\!\!\left\vert
\partial_{x_{i}}w_{t}\right\vert ^{\!2\!}\!\left(  \!\varepsilon
\!+\!\left\vert \partial_{x_{i}}u\right\vert \!+\!\left\vert \partial_{x_{i}%
}v\right\vert \!\right)  ^{\!p_{i}\!-\!2\!}\!\!\leq\!\frac{1}{\alpha
}\!\overset{N}{\underset{i=1}{\sum}}\!\int_{\!E_{t}}\!\!\!h_{i}\!\left(
\!\left\vert \partial_{x_{i}}u\right\vert \!+\!\left\vert \partial_{x_{i}%
}v\right\vert \!\right)  ^{\!\sigma_{i}\!}\!\left\vert \partial_{x_{i}}%
\!w_{t}\right\vert \!w_{t}. \label{eqqqq2}%
\end{equation}
If $\sigma_{i}\geq\frac{p_{i}-2}{2}$ by Young inequality we have
\begin{equation}
\!\!\!\overset{N}{\underset{i=1}{\sum}}\!\int_{\!E_{t}}\!\!\!\left\vert
\partial_{x_{i}}w_{t}\right\vert ^{\!2\!}\!\left(  \!\varepsilon
\!+\!\left\vert \partial_{x_{i}}u\right\vert \!+\!\left\vert \partial_{x_{i}%
}v\right\vert \!\right)  ^{\!p_{i}\!-\!2\!}\!\!\leq\!c_{1}\!\overset
{N}{\underset{i=1}{\sum}}\!\int_{\!E_{t}}\!\!h_{\!i}^{\!2}\!\left(
\!\left\vert \partial_{x_{i}}u\right\vert \!+\!\left\vert \partial_{x_{i}%
}v\right\vert \!\right)  ^{2\sigma_{i}-\left(  p_{i}-2\right)  }w_{t}^{2}
\label{stima 2}%
\end{equation}
for some positive constant $c_{1}$ independent on $t$. Using inequalities
(\ref{sobolev 2}) and (\ref{geo-arti}), H\"{o}lder inequality and
(\ref{stima 2}) we obtain%
\begin{align*}
\!\frac{1}{C_{S}}\!\left(  \!\int_{\!E_{t}}\!w_{t}^{2^{\ast}}\!\right)
^{\!\frac{2}{2^{\ast}}}\!  &  \leq\!\overset{N}{\underset{i=1}{%
{\displaystyle\prod}
}}\!\left(  \!\int_{\!E_{t}}\!\left\vert \partial_{x_{i}}w_{t}\right\vert
^{2}\!\right)  ^{\!\frac{1}{N}}\!\leq\!c_{2}\!\overset{N}{\underset{i=1}{%
{\displaystyle\prod}
}}\!\left(  \!\int_{\!E_{t}}\!\left\vert \partial_{x_{i}}w_{t}\right\vert
^{2}\!\left(  \!\varepsilon\!+\!\left\vert \partial_{x_{i}}u\right\vert
\!+\!\left\vert \partial_{x_{i}}v\right\vert \!\right)  ^{\!p_{i}%
-2\!}\!\right)  ^{\!\frac{1}{N}\!}\!\\
&  \leq c_{2}\overset{N}{\underset{i=1}{\sum}}\int_{E_{t}}h_{i}^{2}\left(
\left\vert \partial_{x_{i}}u\right\vert +\left\vert \partial_{x_{i}%
}v\right\vert \right)  ^{2\sigma_{i}-\left(  p_{i}-2\right)  }w_{t}^{2}\\
&  \leq c_{2}\left(  \int_{E_{t}}w_{t}^{2^{\ast}}\right)  ^{\frac{2}{2^{\ast}%
}}\overset{N}{\underset{i=1}{\sum}}\left(  \int_{E_{t}}h_{i}^{N}\left(
\left\vert \partial_{x_{i}}u\right\vert +\left\vert \partial_{x_{i}%
}v\right\vert \right)  ^{\left(  2\sigma_{i}-p_{i}+2\right)  \frac{N}{2}%
}\right)  ^{\frac{2}{N}},
\end{align*}
i.e.%
\[
\frac{1}{C_{S}}\leq c_{2}\overset{N}{\underset{i=1}{\sum}}\left(  \int_{E_{t}%
}h_{i}^{N}\left(  \left\vert \partial_{x_{i}}u\right\vert +\left\vert
\partial_{x_{i}}v\right\vert \right)  ^{\left(  2\sigma_{i}-p_{i}+2\right)
\frac{N}{2}}\right)  ^{\frac{2}{N}}%
\]
for some\ positive constant $c_{2}$ (independent on $t)$, that can be vary
from line to line. Since $\frac{N}{s_{i}}+\frac{\left(  2\sigma_{i}%
-p_{i}+2\right)  }{p_{i}}\frac{N}{2}\leq1$ we have%
\[
\underset{t\rightarrow\sup w}{\lim}\int_{E_{t}}h_{i}^{N}\left(  \left\vert
\partial_{x_{i}}u\right\vert +\left\vert \partial_{x_{i}}v\right\vert \right)
^{\left(  2\sigma_{i}-p_{i}+2\right)  \frac{N}{2}}=0,
\]
that gives a contradiction.

\noindent Conversely if $\sigma_{i}<\frac{p_{i}-2}{2}$ by (\ref{eqqqq2}) and
Young inequality we have%
\begin{equation}
\!\!\!\overset{N}{\underset{i=1}{\sum}}\!\int_{\!E_{t}}\!\!\!\left\vert
\partial_{x_{i}}w_{t}\right\vert ^{\!2\!}\!\!\!\leq\!c_{3}\!\overset
{N}{\underset{i=1}{\sum}}\!\int_{\!E_{t}}\!\!h_{\!i}^{2\!}\!\left(
\!\left\vert \partial_{x_{i}}u\right\vert \!+\!\left\vert \partial_{x_{i}%
}v\right\vert \!\right)  ^{2\sigma_{i}}w_{t}^{2} \label{stima22}%
\end{equation}
for some positive constant $c_{3}$ independent on $t$. Using inequalities
(\ref{sobolev 2}) and (\ref{geo-arti}), H\"{o}lder inequality and
(\ref{stima22}) as before we obtain%
\[
\frac{1}{C_{S}}\leq c_{4}\overset{N}{\underset{i=1}{\sum}}\left(  \int_{E_{t}%
}h_{i}^{N}\left(  \left\vert \partial_{x_{i}}u\right\vert +\left\vert
\partial_{x_{i}}v\right\vert \right)  ^{N\sigma_{i}}\right)  ^{\frac{2}{N}}%
\]
for some\ positive constant $c_{4}$ independent on $t$. Since $\frac{N}{s_{i}%
}+\frac{\sigma_{i}N}{p_{i}}\leq1$ the assert.
\end{proof}

\bigskip

\begin{remark}
We observe that for $p_{i}>\frac{2Ns_{i}}{Ns_{i}-2s_{i}-2N}$, Theorem
\ref{th+H2} holds for $0\leq\sigma_{i}\leq\frac{p_{i}}{N}-\frac{p_{i}}{s_{i}}$
or $\frac{p_{i}-2}{2}\leq\sigma_{i}\leq\frac{p_{i}}{N}-\frac{p_{i}}{s_{i}%
}+\frac{p_{i}-2}{2}.$ Moreover Theorems \ref{Th+H1} and \ref{th+H2} hold if in
(\ref{P1}) we add the term $c(x,u)$ as in Remark \ref{remark b}.
\end{remark}

\end{document}